\documentclass[10pt,reqno]{amsart}
\RequirePackage[OT1]{fontenc}
\RequirePackage{amsthm,amsmath}
\RequirePackage[numbers]{natbib}
\RequirePackage[colorlinks,linkcolor=blue,citecolor=blue,urlcolor=blue]{hyperref}
\usepackage{amssymb}
\usepackage{anysize}
\usepackage{hyperref}
\usepackage{extarrows}
\usepackage{multirow}
\usepackage{natbib}
\usepackage{stmaryrd}
\usepackage{color}
\usepackage{amsmath}
\usepackage{enumitem}
\usepackage{mathtools} 
\usepackage{dsfont} 

\usepackage{soul}

\renewcommand{\Re}{\mathrm{Re}\,}
\renewcommand{\Im}{\mathrm{Im}\,}

\newcommand{\im}{\mathrm{Im}\,}
\newcommand{\E}{{\mathbb E }}
\newcommand{\R}{{\mathbb R }}
\newcommand{\N}{{\mathbb N}}

\renewcommand{\P}{{\mathbb P}}
\newcommand{\C}{{\mathbb C}}
\newcommand{\ii}{\mathrm{i}}
\newcommand{\deq}{\mathrel{\mathop:}=}

\newcommand{\dd}{\mathrm{d}}

\newcommand{\wt}{\widetilde}
\newcommand{\bs}{\boldsymbol}

\def\ea{e_{0}(\lambda)}
\def\Tr{\mathrm{Tr}}
\def\i{\text{i}}

\def\pzp{\frac{\partial}{\partial z'}}
\def\pzz{\frac{\partial}{\partial \overline{z}}}

\def\px{\frac{\partial}{\partial x}}
\def\py{\frac{\partial}{\partial y}}
\def\tf{\tilde{f}}

\def\ga{G^{(1)}}
\def\gb{G^{(2)}}
\def\Ttwo{T(z,z')}
\def\Tone{T(z,z)}
\renewcommand{\mathbf}[1]{\bs{#1}}
\def\<{\langle}
\def\>{\rangle}

\numberwithin{equation}{section}
\theoremstyle{plain}
\newtheorem{theorem}{Theorem}[section]
\newtheorem{lemma}[theorem]{Lemma}
\newtheorem{proposition}[theorem]{Proposition}
\newtheorem{definition}[theorem]{Definition}
\newtheorem{remark}[theorem]{Remark}
\newtheorem{assumption}[theorem]{Assumption}

\marginsize{25mm}{25mm}{25mm}{26mm}  

\begin{document}

\begin{center}

 \begin{minipage}{0.70\textwidth}
 \vspace{2.5cm}
 
\begin{center}
\large\bf
On fluctuations of global and mesoscopic linear statistics of generalized Wigner matrices
\end{center}
\end{minipage}
\end{center}

\renewcommand{\thefootnote}{\fnsymbol{footnote}}	
\vspace{1cm}
\begin{center}
 \begin{minipage}{1.1\textwidth}

 \begin{minipage}{0.42\textwidth}
\begin{center}
Yiting Li  \\
		\footnotesize 
		{Aix-Marseille Universit\'e, CNRS}\\
{\it yiting.li@univ-amu.fr}
\end{center}
\end{minipage}
\begin{minipage}{0.05\textwidth}
	\begin{center}
	\quad\\
		\footnotesize 
		{ \text{}}\\
		{\it }
	\end{center}
\end{minipage}
\begin{minipage}{0.42\textwidth}
\begin{center}
Yuanyuan Xu\\
\footnotesize 
		{KTH Royal Institute of Technology}\\
{\it yuax@kth.se}
\end{center}
\end{minipage}
\end{minipage}
\end{center}

\renewcommand{\thefootnote}{\fnsymbol{footnote}}	

\vspace{1cm}

\begin{center}
 \begin{minipage}{0.83\textwidth}\footnotesize{
 {\bf Abstract.}
} 
We consider an $N$ by $N$ real or complex generalized Wigner matrix $H_N$, whose entries are independent centered random variables with uniformly bounded moments. We assume that the variance profile, $s_{ij}:=\mathbb{E} |H_{ij}|^2$, satisfies $\sum_{i=1}^Ns_{ij}=1$, for all $1 \leq j \leq N$ and $c^{-1} \leq N s_{ij} \leq c$ for all $ 1 \leq i,j \leq N$ with some constant $c \geq 1$. We establish Gaussian fluctuations for the linear eigenvalue statistics of $H_N$ on global scales, as well as  on all mesoscopic scales up to the spectral edges, with the expectation and variance formulated in terms of the variance profile. We subsequently obtain the universal mesoscopic central limit theorems for the linear eigenvalue statistics inside the bulk and at the edges respectively.
\end{minipage}
\end{center}

 \vspace{5mm}
 
 {\small
\footnotesize{\noindent\textit{Date}: \today}
 
 \vspace{2mm}

 }

\section{Introduction}

\subsection{Linear eigenvalue statistics of Wigner matrices}

A Wigner matrix $H_N$ is an $N \times N$ matrix whose entries are independent real or complex valued random variables up to the symmetry constraint $H_N=H_N^*$. Wigner matrices with real or complex Gaussian entries are known as the $Gaussian$ $Orthogonal$ $Ensemble$ (GOE) and the $Gaussian$ $Unitary$ $Ensemble$ (GUE), respectively. The celebrated Wigner semicircle law \cite{wigner_annals} states that the empirical eigenvalue distribution of $H_N$ converges to the semicircle distribution with density $\rho_{sc}(x):=\frac{1}{2 \pi} \sqrt{4-x^2} \mathds{1}_{[-2,2]}$. More precisely, denoting by $(\lambda_i)_{i=1}^N$ the eigenvalues of $H_N$, for any sufficiently regular test function $f$, the linear statistics
$\frac{1}{N}\sum_{i=1}^N f(\lambda_i) - \int_{\R}f(x) \rho_{sc}(x)\dd x$ converges in probability to zero as $N\to \infty$, which can be understood as a Law of Large Numbers.

It is then natural to derive the corresponding Central Limit Theorem (CLT), i.e., the Gaussian fluctuations of the {\it linear eigenvalue statistics}
\begin{align}\label{linear_statistics_for_Wigner}
\sum\limits_{i=1}^Nf(\lambda_i)-\E\Big[ \sum\limits_{i=1}^Nf(\lambda_i)\Big].
\end{align}
The linear statistics (\ref{linear_statistics_for_Wigner}) need not be normalized by $N^{-\frac{1}{2}}$ as in the classical CLT, which can be explained by the strong correlations among eigenvalues. Khorunzhy,  Khoruzhenko and Pastur \cite{KKP_CLT} proved a CLT for the trace of the resolvent of Wigner matrices. Johansson \cite{Johansson} derived Gaussian fluctuations for the linear eigenvalue statistics of invariant ensembles, including the GUE and GOE. Bai and Yao \cite{baiwigner} used a martingale method to extend the CLTs to arbitrary Wigner matrices and analytic test functions. The regularity conditions on the test functions were weakened by Lytova and Pastur \cite{lytova+pastur}, Shcherbina \cite{M.Shcherbina} via the characteristic function of (\ref{linear_statistics_for_Wigner}), and more recently by Sosoe and Wong~\cite{SosoeWong} who obtained the CLT for $H^{1+\epsilon}$ test functions.

The fluctuations of the linear eigenvalue statistics on mesoscopic scales, i.e.,
\begin{align}\label{eq1}
\sum\limits_{i=1}^Nf\Big(\frac{\lambda_i-E_0}{\eta_0}\Big)-\mathbb E\Big[\sum\limits_{i=1}^Nf\Big(\frac{\lambda_i-E_0}{\eta_0}\Big)\Big],
\end{align}
with fixed energy $E_0 \in (-2,2)$ and scale parameter $N^{-1} \ll \eta_0 \ll 1$, were first studied by Boutet de Monvel and Khorunzhy \cite{meso1} for the GOE given the test function $f(x)=(x-\ii)^{-1}$. They subsequently extended their results to real Wigner matrices \cite{meso2} with $N^{-\frac{1}{8}} \ll \eta_0 \ll 1$. A Mesoscopic CLT for the GUE was obtained by Fyodorov, Khoruzhenko and Simm \cite{FKS}, and was extended by Lodhia and Simm \cite{mesowigner} to complex Wigner matrices on scales $N^{-1/3}\ll\eta_0\ll1$. He and Knowles \cite{moment} improved these CLTs on optimal mesoscopic scales $N^{-1}\ll\eta_0\ll1$ for all Wigner matrices. They also studied the two point correlation function of Wigner matrices on mesoscopic scales in \cite{HeKnowlesDensityCorrelation}. More recently, Landon and Sosoe \cite{character} obtained similar CLTs by studying the characteristic function of~(\ref{eq1}).

Mesoscopic linear eigenvalue statistics can also be studied at the spectral edges, where the mesoscopic scale ranges over $N^{-\frac{2}{3}} \ll \eta_0 \ll 1$. Basor and Widom \cite{basor+widom} used asymptotics of the Airy kernel to derive Gaussian fluctuations of the linear eigenvalue statistics of the GUE at the edges. Min and Chen \cite{Min+Chen} subsequently extended this result to the GOE. Adhikari and Huang \cite{Adhikari+Huang} proved the mesoscopic CLT for the Dyson Brownian motion at the edges down to the optimal scale $\eta_0 \gg N^{-\frac{2}{3}}$ in a short time. Recently, Schnelli and the authors \cite{Li+Schnelli+Xu} obtained mesoscopic CLT for deformed Wigner matrices at regular edges, where the spectral density has square-root behaviors. 

Besides Wigner matrices, mesoscopic CLTs were also obtained in many other random matrices ensembles, e.g., random band matrices \cite{erdos+knowles,erdos+knowles2}, sparse Wigner matrices~\cite{sparse}, Dyson Brownian motion \cite{Duits+Johansson,Huang+Landon,character2}, invariant $\beta$-ensembles \cite{Bekerman+lodhia, BEYY, lambert_1}, orthogonal polynomial ensembles \cite{Breuer+Duits}, classical compact groups \cite{Soshnikov}, circular $\beta$ ensembles \cite{lambert_2}, and free sum of matrices \cite{Bao+Schnelli+Xu}.  

\subsection{Generalized Wigner matrices}
In this paper, we are interested in the linear eigenvalue statistics for generalized Wigner matrices, which were introduced in \cite{Erdos+Yau+Yin}. Let $H_N=(H_{ij})_{i,j=1}^N$ be an $N$ by $N$ matrix with independent but not identically distributed centered random variables up to the symmetry constraint $H_{N}=H_N^*$. Denote by $S \equiv S_N$ the matrix of variances, i.e. $S:=(s_{ij})_{i,j=1}^N$, with $s_{ij}=\E |H_{ij}|^2$. We assume that $S$ is symmetric and doubly stochastic, i.e.,
\begin{equation}\label{normal}
\sum_{i=1}^N s_{ij}=1, \quad \mbox{for all }1 \leq j \leq N.
 \end{equation}
 We say $H_N$ is a generalized Wigner matrix if the size of $s_{ij}$ is comparable with $N^{-1}$, that is, there exists $c\geq1$ independent of $N$ such that 
 	\begin{equation}\label{flatness}
 	\quad c^{-1}\leq N s_{ij} \leq c,\quad\text{for all } 1 \leq i,j \leq N.
 	\end{equation}  
Standard Wigner matrices are a special case of generalized Wigner matrices, with $s_{ij}=N^{-1}$ for all $1 \leq i,j \leq N$. The first condition in (\ref{normal}) guarantees that the limiting spectral measure of $H_N$ is given by the semicircle law; see \cite{Anderson+Zeitouni,Guionnet, DK}. Without the condition (\ref{normal}), the limiting eigenvalue distribution is characterized by the Dyson equation and were classified in \cite{wigner_type_1}. Local laws of such general Wigner-type matrices were obtained in \cite{wigner_type_2, wigner_type_3} and bulk universality was then established in \cite{wigner_type_2}, while the edge and cusp universality were derived in \cite{wigner_type_3, isotropic3, Cusp1}.

The second assumption (\ref{flatness}) is a sufficient condition for generalized Wigner matrices to demonstrate the same local eigenvalue statistics as standard Wigner matrices. Universality for the local eigenvalue statistics of generalized Wigner matrices was obtained in \cite{BEYY,Erdos+Yau+Yin2,Erdos+Yau+Yin3} for the bulk and in \cite{BEY Edge,Erdos+Yau+Yin,Lee+Yin} for the edges. For random band matrices, the condition (\ref{flatness}) is not satisfied. We refer to \cite{BYJ_1, BYJ_2, Erdos+Knowles+Yau+Yin} for results on local laws and bulk universality, and to \cite{sodin} for edge universality.

Consider now a special variance matrix $S$ with $s_{ij}=\frac{1}{N}f\big(\frac{i}{N}, \frac{j}{N} \big)$, where $f \in C([0,1]\times [0,1])$ is a non-negative, symmetric function such that $\int_{0}^1 f(x,y)dy \equiv 1$.  A CLT for the linear eigenvalue statistics of such matrices was obtained in \cite{Anderson+Zeitouni} by studying its generating function via combinatorial enumeration, with the variance formulated as an infinite series. Global CLTs for random band matrices were obtained in \cite{Li+Sasha, jana1, Shcherbinaband}, while the mesoscopic linear statistics were studied in \cite{erdos+knowles,erdos+knowles2}.  Fluctuations of the linear eigenvalue statistics on global scales for many familiar classes of random matrices were also studied in \cite{poincare}, where a unified technique was formulated for deriving such CLTs using second order Poincar\'e inequalities, without an explicit formula for the variance. Under this framework, CLTs for linear eigenvalue statistics of Wigner matrices with general variance profiles were obtained in \cite{jana2}. Global fluctuations of block Gaussian matrices with variance profiles were proved within the framework of second-order free probability theory, see \cite{block_gaussian} and references therein. In addition, CLTs on global scales for large sample covariance matrices given a general variance profile were discussed in \cite{najim3}.

In the present paper, we consider generalized Wigner matrices with matrix of variances $S$ satisfying (\ref{normal}) and (\ref{flatness}). We derive Gaussian fluctuations for the linear eigenvalue statistics~\eqref{eq1}, with explicit integral formulas for the variance and expectation in terms of the matrix of variances $S$, at fixed energy~$E_0 \in [-2,2]$ on scales $N^{-1} \ll \eta_0 \leq 1$ such that $ \eta_0\sqrt{\eta_0+\kappa_0} \gg N^{-1}$, where $\kappa_0=\kappa_0(E_{0})$ denotes the distance from $E_0$ to the closest edge of the semicircle law; see Theorem~\ref{thm:weak_convergence}. This range of $\eta_0$ covers the global scales as well as all mesoscopic scales up to the spectral edges. Furthermore, we obtain the universal CLTs on all mesoscopic scales, for energies $E_0$ in the bulk and at the edges respectively, by computing the variances and expectations explicitly considering mesoscopic-scaled test functions; see Theorem \ref{meso}. The limiting law is universal, only depending on the symmetry class, and is independent of the scaling $\eta_0$ and the energy $E_0$.

The proof of our main technical result Proposition \ref{prop} is provided in Section \ref{sec:proof_of_proposition}. We follow the idea of~\cite{lytova+pastur, character} to study the characteristic function of the linear eigenvalue statistics~\eqref{eq1}. Via the Helffer-Sj\"ostrand functional calculus, we write the derivative of the characteristic function in terms of the resolvent of $H_N$, and then cut off the ultra-mesoscopic scales of the spectral domain, see (\ref{phi2}), since the very local scales do not contribute to the mesoscopic linear statistics. The benefit is that on the restricted spectral domain, the resolvent of $H_N$ is controlled effectively by the local laws \cite{locallaw,Erdos+Yau+Yin}. We subsequently apply the cumulant expansions (see Lemma \ref{cumulant}) to solve the right side of (\ref{phi2}). This technique was first used in random matrix theory by \cite{KKP_CLT} and in recent papers, e.g., \cite{isotropic2,moment, edge, lytova+pastur}. The key tools to estimate the error in Proposition \ref{prop} are the (isotropic) local laws for the resolvent \cite{Alex+Erdos+Knowles+Yau+Yin,isotropic3, isotropic2, yukun_antti} and the fluctuation averaging estimates \cite{Erdos+Knowles+Yau,Erdos+Knowles+Yau+Yin, short_proof,YY}. One of the main technical achivements is to find a weak local law for the two point function $T_{ab}(z,z'):=\sum_{j=1, j\neq b}^N s_{aj} G_{jb}(z) G_{jb}(z')$, with different spectral parameters $z,z'$; see Lemma \ref{tracet} with proof in Section \ref{sec:two_point_function}. Compared with the standard Wigner matrices \cite{moment, character}, the two point function $T_{ab}(z,z')$ cannot be written as a matrix product and hence the resolvent identity (\ref{resolvent}) or cyclicity of trace no longer help. Similar two point functions of the resolvents appeared in \cite{erdos_reference_1, erdos_reference_2, non_hermitian, Bao+Schnelli+Xu} to derive Gaussian fluctuations of the linear eigenvalue statistics for different random matrix ensembles. The proof of Lemma \ref{tracet} is inspired by the fluctuation averaging mechanism \cite{Erdos+Knowles+Yau}, combined with recursive moment estimates based on cumulant expansions. A special case $z=\bar{z}$ was studied previously in \cite{Erdos+Knowles+Yau, short_proof, YY}, and our statements are for arbitrary parameters $z,z' \in \C \setminus \R$. In addition, we end Section \ref{sec:proof_of_proposition} by estimating the expectation of the linear eigenvalue statistics and then complete the proof of Theorem \ref{thm:weak_convergence}.

{\it Notation:} We will use the following definition on high-probability estimates from~\cite{Erdos+Knowles+Yau}. 
\begin{definition}\label{definition of stochastic domination}
Let $\mathcal{X}\equiv \mathcal{X}^{(N)}$ and $\mathcal{Y}\equiv \mathcal{Y}^{(N)}$ be two sequences of nonnegative random variables. We say~$\mathcal{Y}$ stochastically dominates~$\mathcal{X}$ if, for all (small) $\epsilon>0$ and (large)~$D>0$,
\begin{align}\label{prec}
\P\big(\mathcal{X}^{(N)}>N^{\epsilon} \mathcal{Y}^{(N)}\big)\le N^{-D},
\end{align}
for sufficiently large $N\ge N_0(\epsilon,D)$, and we write $\mathcal{X} \prec \mathcal{Y}$ or $\mathcal{X}=O_\prec(\mathcal{Y})$.
\end{definition}
We often use the notation $\prec$ also for deterministic quantities, then~\eqref{prec} holds with probability one. Properties of stochastic domination can be found in Lemma \ref{dominant}.

For any vector $\mathbf v \in \C^{N}$, let $\|\mathbf v\|_{\sup}:=\max_{i=1}^N|v_i|$ be the sup norm. For any matrix $A \in \C^{N \times N}$, the matrix norm induced by the sup vector norm are given by $\|A\|_{\infty}:=\max_{1\leq i\leq N}\sum_{j=1}^N|A_{ij}|$. We also write $\|A\|_{\sup}:=\max_{i,j}|A_{ij}|$.

Throughout the paper, we use~$c$ and~$C$ to denote strictly positive constants that are independent of $N$. Their values may change from line to line. We use standard Big-O and little-o notations. For $X,Y \in \R$, we write $X \ll Y$ if there exists a small $\tau>0$ such that $|X| \leq N^{-\tau} |Y|$ for large $N$. Moreover, we write $X \sim Y$ if there exist constants $c, C>0$ such that $c |Y| \leq |X| \leq C |Y|$ for large $N$. Finally, we denote the upper half-plane by $\C^+\deq\{z\in\C\,:\,\im z>0\}$.

\section{Main results}\label{sec:main_results}

Let $H \equiv H_N$ be an $ N \times N$ real or complex generalized Wigner matrix satisfying the following assumption.
\begin{assumption}\label{assumption_1}
For real ($\beta=1$) generalized Wigner matrix, we assume that 
	\begin{enumerate}
		\item $\{ H_{ij}|  i \leq j \} $ are independent real-valued centered random variables with $H_{ij}=H_{ji}$.
		\item Let $S \equiv S_N$ denote the matrix of variances, i.e., $S:=(s_{ij})_{i,j=1}^N$ with $s_{ij}=\E |H_{ij}|^2$. There exist constants $0 <C_{\inf} \leq C_{\sup} < \infty$ such that
		\begin{equation}\label{flat}
		\sum_{i=1}^N s_{ij}\equiv 1;\qquad C_{\inf} \leq \inf_{N,i,j} N s_{ij} \leq \sup_{N,i,j} N s_{ij} \leq C_{\sup}.
		\end{equation}
		\item  All moments of the entries of $\sqrt{N}H_N$ are uniformly bounded, i.e., for any $k \in \N$, there exists $C_k$ independent of $N$ such that for all $1 \leq i,j \leq N$,
		\begin{equation}\label{moment_condition}
		\E |\sqrt{N} H_{ij}|^k \leq C_k.
		\end{equation}	
	\end{enumerate}
For complex ($\beta=2$) generalized Wigner matrix, we assume that
	\begin{enumerate}
		\item[(a)] $\{\Re H_{ij},\Im H_{ij}| i \leq j \}$ are independent real-valued centered random variables with $H_{ij}=\overline{H_{ji}}$. 
		\item[(b)] The same moment conditions 2 and 3 hold and $\E [H_{ij}^2]=0$ for $i \neq j$. 
	\end{enumerate}
\end{assumption}

For a probability measure $\nu$ on $\R$, denote by $m_\nu$ its Stieltjes transform, i.e.,
\begin{align}
 m_\nu(z)\deq\int_\R\frac{\dd\nu(x)}{x-z}\,,\qquad z\in\C^+\,.
\end{align}
Note that $m_{\nu}\,:\C^+\rightarrow\C^+$ is analytic and can be analytically continued to the real line outside the support of $\nu$. Moreover, $m_{\nu}$ satisfies $\lim_{\eta\nearrow\infty}\ii\eta {m_{\nu}}(\ii \eta)=-1$. The Stieltjes transform of the semicircle law $\mu_{sc}:=\rho_{sc}(x) \dd x=\frac{1}{2 \pi} \sqrt{4-x^2} \mathds{1}_{[-2,2]} \dd x$, denoted by $m_{sc}$, is defined as the unique analytic solution $\C^+ \rightarrow \C^+$ satisfying 
\begin{equation}\label{msc}
m_{sc}^2(z)+zm_{sc}(z)+1=0.
\end{equation}

Fix the energy $E_0 \in [-2, 2]$ and set $ N^{-1} \ll \eta_0 \leq 1$. Consider a scaled test function 
\begin{equation}\label{fn}
f \equiv f_{N}(x):=g\Big( \frac{x-E_0}{\eta_0}\Big),\qquad g \in C_c^2(\R).
\end{equation}
Define the distance between the support of $f$ and the nearest edge of the semicircle~law,
\begin{equation}\label{kappa_0}
\kappa_0:=\mbox{dist}(\mbox{supp}(f),\{-2, 2\}).
\end{equation}
Then we have the following CLT for the linear eigenvalue statistics of $H_N$.
\begin{theorem}\label{thm:weak_convergence}
	Let $H_N$ be a generalized Wigner matrix satisfying Assumption \ref{assumption_1} and assume that $ \eta_0 \sqrt{\kappa_0 +\eta_0} \geq N^{-1+c_0}$ for some constant $c_0>0$. Then there exists a small constant $0<\tau<\frac{c_0}{16}$ such that the following statements hold. For $f$ as in (\ref{fn}), define
	\begin{align}\label{vf}
	V(f):=&-\frac{1}{4 \pi^2} \int_{\Gamma_1} \int_{\Gamma_2} \tf(z) \tf(z') \Big\{ \frac{2}{\beta} \Tr \Big(  \frac{m'_{sc}(z) m'_{sc}(z') S}{(1-m_{sc}(z) m_{sc}(z') S)^2}\Big)\nonumber\\
	& +2 k_4 m_{sc}(z)m'_{sc}(z) m_{sc}(z')m'_{sc}(z')+\Tr S \Big(1-\frac{2}{\beta}\Big) m'_{sc}(z)m'_{sc}(z') \Big\} \dd z \dd z',
	\end{align}
where
\begin{itemize}
	\item $k_4$ is the summation of the forth cumulants (see (\ref{cumulant_formula}) and (\ref{K_4})) of both real and imaginary parts of all entries $\{H_{ij}\}$;
	\item $\tf$ is an almost-analytic extension of $f$, i.e.
	\begin{equation}\label{tilde_f}
	\tilde{f}(x+\ii y):=(f(x)+\ii y f'(x)) \chi(y),
	\end{equation}
	where  $\chi: \mathbb R\to[0,1]$ is a smooth cutoff function with support in $[-2,2]$ and with $\chi(y)=1$, for $|y| \leq 1$;
	\item the contours $\Gamma_{k}~(k=1,2)$ are given by $\{ z \in \C\,:\, |\Im z| = \frac{1}{k }N^{-\tau} \eta_0 \}$ with counterclockwise orientation.
\end{itemize}
 If there exist constants $c, C>0$ such that $c<V(f) < C$, then
	$$\frac{\Tr f(H_N)-\E \Tr f(H_N)}{\sqrt{V(f)}} \xrightarrow{d} \mathcal{N}(0,1).$$
Moreover, the so-called bias is given by
	\begin{align}\label{bf}
	\E \Tr f(H_N)-&N \int_{\R} f(x) \rho_{sc}(x) \dd x= \frac{1}{2 \pi \ii } \int_{\Gamma_1}  \tf(z) \Big\{ \Big( \frac{2}{\beta} -1\Big) \Tr \Big( \frac{m'_{sc}(z) m^3_{sc}(z)  S^2}{1-m^2_{sc}(z)S} \Big)\nonumber\\
	&+k_4 m'_{sc}(z) m^3_{sc}(z)  \Big\}\dd z+O_{\prec}\Big( \frac{N^{2 \tau}}{(N \eta_0 \sqrt{\kappa_0+\eta_0})^{1/4}}\Big)+O_{\prec}(N^{-\tau}).
	\end{align}
	
\end{theorem}

\begin{remark}
We remark that Theorem \ref{thm:weak_convergence} applies to the global scales as well as optimal mesoscopic scales up to the spectral edges. The formulas for the variance (\ref{vf}) and the bias (\ref{bf}) coincide with the corresponding results for standard Wigner matrices \cite{character,Li+Schnelli+Xu} where $s_{ij}=N^{-1}$ for all $1 \leq i,j\leq N$.
\end{remark}

Finally, we obtain the following mesoscopic CLTs for the linear eigenvalue statistics in the bulk and at the edges respectively.
\begin{theorem}[Universal mesoscopic CLTs]\label{meso}
	Let $H_N$ be a generalized Wigner matrix satisfying Assumption \ref{assumption_1}. Fix any $E_0 \in (-2, 2)$ and $c_1 \in (0,1)$, and set $\eta_0=N^{-c_1}$. For any function $g \in C^2_c(\R)$, the mesoscopic linear statistics in the bulk
	\begin{equation*}
	\sum_{i=1}^N g \Big( \frac{\lambda_i-E_0}{\eta_0} \Big)-N \int_{\R} g \Big(\frac{x-E_0}{\eta_0} \Big) \rho_{sc}(x) \dd x \xrightarrow{d} \mathcal{N} \Big(0,\frac{1}{\beta \pi} \int_{\R} |\xi| |\hat{g}(\xi)|^2 \dd \xi \Big),
	\end{equation*}
	where $\hat{g}(\xi):=(2 \pi)^{-1/2} \int_{\R} g(x) e^{-\ii \xi x} \dd x$.	
	
	In addition, set $E_0=\pm 2$ and $\eta_0=N^{-c_2}$ with any fixed $c_2 \in (0, \frac{2}{3})$. Then the mesoscopic linear statistics at the edges
	\begin{equation*}
	\sum_{i=1}^N g \Big( \frac{\lambda_i-E_0}{\eta_0} \Big)-N \int_{\R} g \Big(\frac{x-E_0}{\eta_0} \Big) \rho_{sc}(x) \dd x \xrightarrow{d} \mathcal{N} \Big(\Big( \frac{2}{\beta} -1\Big)\frac{g(0)}{4},\frac{1}{2 \beta \pi} \int_{\R} |\xi| |\hat{h}(\xi)|^2 \dd \xi \Big),
	\end{equation*}
       where $h(x):=g(\mp x^2)$, and $\hat{h}(\xi):=(2 \pi)^{-1/2} \int_{\R} h(x) e^{-\ii \xi x} \dd x$.
	\end{theorem}

\begin{remark}
The means and variances of the limiting laws in Theorem \ref{meso} agree with the corresponding results for the Gaussian ensembles. See \cite{meso1,FKS, Min+Chen_another_paper} for the bulk and \cite{basor+widom, Min+Chen} for the edges. Such edge formulas were also obtained in other ensembles, e.g., Dyson Brownian motion \cite{Adhikari+Huang}, deformed Wigner matrices and sample covariance matrices~\cite{Li+Schnelli+Xu}. 
\end{remark}

\section{Preliminaries}\label{sec:preliminary}
In this section, we introduce some preliminary results that will be used in the proof. 

\subsection{Properties of the Stieltjes transform of the semicircle law}
In this subsection, we recall some properties of $m_{sc}$. Let $\kappa=\kappa(E)$ be the distance from $E$ to the closest spectral edge of the semicircle law, i.e.,
\begin{equation}\label{kappa}
\kappa:=\min \{ |E+2|, |E-2|  \}.
\end{equation}
Define the spectral domain
\begin{equation}\label{d}
D:=\{z=E+\i \eta:|E| \leq 5, 0< \eta \leq 10 \}.
\end{equation}

\begin{lemma}[Lemma 4.2 in \cite{Erdos+Yau+Yin3}, Lemma 6.2 in \cite{book}]\label{previousgene}
We have the following estimates.
	\begin{enumerate}
		\item For any $z \in D$, there exists a constant $c>0$ such that
		\begin{equation}\label{1} 
		c \leq |m_{sc}(z)| \leq 1-c \eta.
		\end{equation}
		\item For all $z \in D$, we have	
		\begin{equation}\label{22}
		|{\Im} m_{sc}(z)|  \sim \begin{cases}
		\sqrt{\kappa+\eta}, & \mbox{if } |E| \leq 2, \\
		\frac{\eta}{\sqrt{\kappa+\eta}}, & \mbox{otherwise}.
		\end{cases}
		\end{equation}		
		\item For all $z \in D$, there exist some constants $c,C>0$ such that
		\begin{equation}\label{3}
		c \sqrt{\kappa +\eta} \leq | 1-m_{sc}^2(z) | \leq C \sqrt{\kappa +\eta}.
		\end{equation}
	\item For all $z \in D$, we have
	\begin{equation}\label{4}
	|m_{sc}(z)| \sim 1; \quad | m'_{sc}(z) | \sim \frac{1}{\sqrt{\kappa+\eta}}; \quad | m''_{sc}(z) | =O\Big(  \frac{1}{\sqrt{(\kappa+\eta)^3}}\Big).
	\end{equation}
	\end{enumerate}
\end{lemma}

\subsection{Properties of the variance matrix $S$}
In this subsection, we state some properties of the matrix of variances $S$, which is crucial in studying the local laws of the generalized Wigner matrices. Recall that $S=(s_{ij})_{i,j=1}^N$ is the matrix of variances satisfying (\ref{flat}), and $S$ is deterministic, symmetric and doubly stochastic with strictly positive entries. Hence $1$ is the largest eigenvalue, with eigenvector $\mathbf{e}:=N^{-\frac{1}{2}} (1,1, \cdots ,1)^T$. By the Perron-Frobenius Theorem, the largest eigenvalue 1 is simple and all other eigenvalues are strictly less than 1 in absolute value. Define $\delta_{\pm}$ to be the spectral gaps satisfying
$$\mbox{Spec}(S) \subset [-1+\delta_-, 1-\delta_+] \cap \{1\}.$$
It is not hard to show that $$\delta_{\pm} \geq C_{\inf}>0,$$
provided $S$ satisfies (\ref{flat}). Combining with (\ref{1}), $1-m_{sc}(z) m_{sc}(z') S$ is invertible. Thus, we have the following estimates.

\begin{lemma}\label{Pi}
	Define $\Pi:=\mathbf{e}\mathbf{e}^T$ with $\mathbf{e}=N^{-\frac{1}{2}} (1,1, \cdots ,1)^T$. For any $z,z' \in D$ or $z,\overline{z'} \in D$, there exists $C>0$ such that
	$$\Big\| \frac{1}{1-m_{sc}(z) m_{sc}(z') S} \Big\|_{\infty} \leq \frac{C}{|1-m_{sc}(z) m_{sc}(z')|}, \quad \Big\| \frac{1-\Pi}{1-m_{sc}(z) m_{sc}(z') S}\Big\|_{\infty} \leq C,$$
	where the constant $C$ depends on $C_{\inf}$ and $C_{\sup}$ in (\ref{flat}).
\end{lemma}
Similar statements can be found in Lemma 6.3 \cite{book} for $z=z'$, and the proof also applies to two parameters $z,z'$. In particular, we have from (\ref{3}) that for all $z \in D$,
\begin{equation}\label{rho}
\rho:=\Big\|\frac{1}{1-m_{sc}^2(z) S}\Big\|_{\infty} \leq C \Big|\frac{1}{1-m^2_{sc}(z)} \Big|  \sim \frac{1}{\sqrt{\kappa+\eta}}.
\end{equation}
We also have a trivial lower bound, $\rho \geq \frac{1}{|1-m^2_{sc}(z)|} \geq \frac{1}{2}$, since $\mathbf{e}$ is an eigenvector of $S$ and $|m_{sc}(z)| \leq 1$.

\subsection{Local Law for the resolvent of $H_N$}
Denote by $(\lambda_i)_{i=1}^N$ the eigenvalues of $H_N$. We define the empirical spectral measure of $H_N$ by $\mu_N(x):=\frac{1}{N} \sum_{i=1}^N \delta_{\lambda_i}$. The Stieltjes transform of $\mu_{N}$ is then given by
\begin{align}\label{normal_trace}
m_N(z):=\int_{\R} \frac{d \mu_N(\lambda)}{\lambda-z}= N^{-1} \Tr G(z), \quad \mbox{with } G(z):=(H_N-zI)^{-1}, \quad z \in \C \setminus \R.
\end{align}
The function $G(z)$ is referred to as the {\it resolvent} or {\it Green function} of $H_N$. 
The semicircle law states that for any fixed $z$ away from the real line, $m_N(z)$ converges in probability to $m_{sc}(z)$ as $N$ tends to infinity. It can be extended down to the local scales $\Im z \gg N^{-1}$. We introduce the spectral domain, 
\begin{equation}\label{ddd}
D':=\big\{z=E+\ii \eta:  |E| \leq 5, N^{-1+\tau} \leq \eta \leq  10 \big\},
\end{equation}
for any constant $\tau>0$, and define two deterministic control parameters for $z=E+\ii \eta \in \C \setminus \R$,
\begin{equation}\label{control}
\Psi \equiv \Psi(z):=\sqrt{ \frac{\im m_{sc}(z)}{N |\eta|}} +\frac{1}{N |\eta|}\,,\qquad \Theta \equiv \Theta(z):=\frac{1}{N |\eta|}\,.
\end{equation}
With estimates of $m_{sc}(z)$ in Lemma \ref{previousgene}, it is easy to check 
\begin{equation}\label{psibound}
C N^{-\frac{1}{2}} \leq \Psi(z) \ll 1\,,\qquad\quad  z \in D'\,.
\end{equation}
We have the following (isotropic) local laws for the resolvent of $H_N$, which is an essential tool in our proof.
\begin{theorem}[Theorem 2.3 in \cite{locallaw}, Theorem 2.12 in \cite{Alex+Erdos+Knowles+Yau+Yin},  (3.8) in \cite{Erdos+Knowles+Yau}]\label{locallawgene}
	Let $H_N$ be a generalized Wigner matrix satisfying Assumption \ref{assumption_1}. The following estimates hold uniformly in $z \in D'$:
	\begin{equation}\label{G}
	\max_{i,j} | G_{ij}(z) -\delta_{ij} m_{sc}(z) |   \prec \Psi(z);\qquad | m_N(z) -m_{sc}(z) | \prec \Theta(z).
	\end{equation}
	Furthermore, we also have for all $z \in D'$,
	\begin{equation}\label{strong}
	\max_{i}\Big| \sum_{j=1}^N s_{ij} G_{jj}(z)-m_{sc}(z) \Big| \prec \rho \Psi^2(z),
	\end{equation}
	with $\rho$ in (\ref{rho}). For any deterministic unit vectors $\mathbf{v},\mathbf w \in \C^N$ and all $z \in D'$, we~have
	\begin{equation}\label{isotropicgene}
	\Big| \langle \mathbf v, G(z) \mathbf w \rangle  - m_{sc}(z) \langle \mathbf v,  \mathbf w \rangle \Big| \prec \Psi(z).
	\end{equation}
\end{theorem}

Finally, we end this section with properties of stochastic domination defined in (\ref{prec}). 
\begin{lemma}[Proposition 6.5 in \cite{book}]\label{dominant}
	\begin{enumerate}
		\item $X \prec Y$ and $Y \prec Z$ imply $X \prec Z$;
		\item If $X_1 \prec Y_1$ and $X_2 \prec Y_2$, then $X_1+X_2 \prec Y_1+Y_2$ and $X_1X_2 \prec Y_1Y_2;$
		\item If $X \prec Y$, $\E Y \geq N^{-c}$ and $|X| \leq N^c$ almost surely with some fixed exponent $c$, then we have $\E X \prec \E Y$.
	\end{enumerate}
\end{lemma}

\section{Proof of Theorem \ref{thm:weak_convergence} and \ref{meso}}\label{sec:proof_of_proposition}
We define the characteristic function of the linear eigenvalue statistics 
\begin{equation}\label{mme}
\phi(\lambda):=\E[e(\lambda)], \quad \mbox{where} \quad e(\lambda):=\exp \Big\{ \i \lambda(\Tr f(H_N)-\E \Tr f(H_N)) \Big\}, \quad \lambda \in \R.
\end{equation}
Then the characteristic function $\phi$ satisfies the following proposition.
\begin{proposition}\label{prop}
	Under the same conditions as in Theorem \ref{thm:weak_convergence}, if $ \eta_0 \sqrt{\kappa_0 +\eta_0} \geq N^{-1+c_0}$ for some $c_0>0$, then there exists a constant $0<\tau<\frac{c_0}{16}$ such that for any fixed $\lambda \in \R$, the characteristic function $\phi(\lambda)$ satisfies
	$$\phi'(\lambda)=-\lambda \phi(\lambda) V(f)+O_{\prec}( |\lambda| \log N N^{- \tau})+O_{\prec} \Big( \frac{(1+|\lambda|^4)N^{4 \tau}}{(N \eta_0 \sqrt{\kappa_0+\eta_0})^{\frac{1}{4}} } \Big).$$
	\end{proposition}
	
	Admitting Proposition \ref{prop}, integrating $\phi'(\lambda)$ and applying the Arzel\'a-Ascoli theorem and L\'evy's continuity theorem, we prove the Gaussian fluctuations for the linear statistics, as stated in Theorem \ref{thm:weak_convergence}. Given the scaled test function (\ref{fn}), we compute the variances (\ref{vf}) and biases (\ref{bf}) on mesoscopic scales in the bulk and at the edges respectively, and then conclude Theorem \ref{meso}. Similar arguments for deformed Wigner matrices can be found in Section 6 \cite{Li+Schnelli+Xu}  and we omit them in the present paper.

	\begin{proof}[Proof of Proposition \ref{prop}]
Via the Helffer-Sj\"ostrand functional calculus (see (4.10) of \cite{character} for a reference), we translate the linear eigenvalue statistics of $f(H_N)$ to the Green function of $H_N$. More precisely, for any $f$ in (\ref{fn}), 
\begin{equation}\label{fw2}
\Tr f(H_N) =\frac{1}{ \pi}  \int_{\C}  \pzz \tf(z) \Tr(G(z)) \dd^2z,
\end{equation}
where $\pzz=\frac{1}{2}(\px+\ii \py)$, $\tf(z)$ is an almost-analytic extension of $f$ given in (\ref{tilde_f}) and $\dd^2z$ is the Lebesgue measure on $\C$. As observed in \cite{character}, the ultra-local scales do not contribute to the mesoscopic linear statistics. So we restrict the domain of the spectral parameter~to
\begin{equation}\label{domain}
\Omega_{0} := \left\{z \in \C\,:\, |\Im z| \geq N^{-\tau} \eta_0  \right\}\,,
\end{equation}
for a small constant $\tau>0$.
Notice that $G(z)$ is analytic in $\C \setminus \R$. Taking derivative of the characteristic function $\phi(\lambda)$ in (\ref{mme}) and applying Stokes' formula, we have
\begin{equation}\label{phi2}
\phi'(\lambda)=\frac{1}{2\pi}\int_{\Gamma_{1}}  \tf(z)  \E \Big[ \ea (\Tr(G(z)-\E \Tr G(z)) \Big]\dd z+O_{\prec}\big( |\lambda| \log N N^{ -\tau} \big),
\end{equation}
where
\begin{equation}\label{e2}
\ea:=\exp\Big\{ \frac{ \lambda}{ 2\pi}   \int_{\Gamma_{2}}  \tf(z') (\Tr(G(z'))-\E \Tr G(z')) \dd z'  \Big\},
\end{equation}
with $\Gamma_{1,2}$ given in Theorem \ref{thm:weak_convergence}. More details can be found in \cite{character, Li+Schnelli+Xu}.  Here we choose slightly different contours to avoid singularities. Thus, in order to study $\phi'(\lambda)$, it suffices to estimate $\E [ (\ea (\Tr(G(z))-\E \Tr G(z)) ]$.

To simplify the proof, we will only consider the real symmetric case $(\beta=1)$. The proof for the complex case $(\beta=2)$ is similar. Before we proceed the proof, we introduce the following cumulant expansion formula, see \cite{moment} for a reference.
\begin{lemma}[Cumulant expansion formula]\label{cumulant}
	Let $h$ be a real-valued random variable with finite moments, and $f$ is a complex-valued smooth function on $\R$ with bounded derivatives. Let $c^{(k)}(h)$ be the $k$-th cumulant of $h$ given by
	\begin{equation}\label{cumulant_formula}
	c^{(k)}(h):=(-\ii)^{k}\frac{\dd^k}{\dd t^k}\Big( \log \E [e^{\ii t h}]\Big) \Big|_{t=0}.
	\end{equation} 
	Then for any fixed $l \in \N$, we have
	\begin{equation}\label{ehh}
	\E[ h f(h)]=\sum_{k=0}^l \frac{1}{k!} c^{{(k+1)}}(h) \E \Big[\frac{\dd^k}{\dd h^k}f(h)\Big] +R_{l+1},
	\end{equation}
	where the error $R_{l+1}$ satisfies
	\begin{equation}\label{error_cumulant}
	|R_{l+1}| \leq C_l \E \Big[|h|^{l+2}\Big] \sup_{|x| \leq M} |f^{(l+1)}(x)| +C_l \E \Big[ |h|^{l+2} 1_{|h>M|}\Big] \sup_{x \in \R} |f^{(l+1)}(x)| ,
       	\end{equation}
	and $M>0$ is an arbitrary fixed cutoff.
\end{lemma}

By the definition of the resolvent and applying the cumulant expansion (\ref{cumulant_formula}) for $l=3$, we have  (c.f. (5.6) in \cite{Li+Schnelli+Xu})
\begin{align}\label{sum}
z\E [\ea(\Tr G-\E \Tr G)]=\E [\ea(\Tr (HG)-\E \Tr (HG))]=I_1+I_2+I_3+R_4,
\end{align}
where
$I_{k} (k=1,2,3)$ denote the expansion terms associated with the $(k+1)$-cumulant, and $R_4$ is the error given in (\ref{error_cumulant}). In the following, we estimate each term on the RHS of (\ref{sum})~using the identity
	\begin{equation}\label{dH}
	\frac{\partial G_{ij}}{\partial H_{ab}}=-\frac{G_{ia} G_{bj}+G_{ib}G_{aj}}{1+\delta_{ab}},\qquad 1 \leq a,b \leq N.
	\end{equation}
	Similar arguments for deformed Wigner matrices can also be found in Section~5~\cite{Li+Schnelli+Xu}, so we omit most computation details. First, it is not hard to check that $R_4=O_{\prec}(N^{-1/2}(1+|\lambda|^4))$, by using the local law (\ref{G}), the moment condition (\ref{moment_condition}) and Lemma \ref{dominant}.

Next, we look at the first term $I_1$. Using the local law (\ref{strong}), $I_1$ can be written as (c.f. Section 5.1 and Lemma 5.1 \cite{Li+Schnelli+Xu})
\begin{align}\label{I_1}
I_1=& -\sum_{i=1}^N \E   [\ea (1-\E) \Big( \sum_{j=1}^N s_{ij} G_{jj}(z)  G_{ii}(z) \Big)] - \sum_{i=1}^N  \E \Big[\ea (1-\E) \Big(  \sum_{j=1}^{(i)} s_{ij}  G_{ji}(z) G_{ji}(z)\Big)  \Big]\nonumber\\
	&-\frac{ \lambda}{\pi} \sum_{i=1}^N  \E \Big[ \ea  \int_{\Gamma_{2}} \tf(z')  \frac{\partial}{\partial z'}   \Big( \sum_{j=1}^{N} \frac{1}{1+\delta_{ij}} s_{ij} G(z')_{ji}  G_{ji}(z) \Big) \dd z' \Big]\nonumber\\
	=&-2 m_{sc}(z) \E [\ea (1-\E)\Tr G] -\E [\ea (1-\E )\Tr \Tone]\nonumber\\
	&-\frac{\lambda}{\pi} \E [\ea \int_{\Gamma_{2}} \tf(z')  \frac{\partial}{\partial z'} \Tr \Ttwo \dd z']-\frac{ \lambda\Tr S }{2 \pi } \E [\ea] \int_{\Gamma_{2}} \tf(z') m_{sc}(z) m'_{sc}(z') \dd z' \nonumber\\
	&+O_{\prec}\big(N \rho \Psi^3(z) \big)+O_{\prec}\Big( \frac{|\lambda|}{\sqrt{N \eta_0}}\Big)+O_{\prec}(|\lambda| \Psi (z)),
	\end{align}
where  we define the following two point function for short,
	\begin{equation}\label{t_function}
	T_{ab}(z,z'):=\sum_{j=1}^{(b)} s_{aj} G_{jb}(z) G_{jb}(z'), \quad 1 \leq a,b \leq N, \quad z,z' \in \C \setminus \R,
	\end{equation}
	with $\sum_{j=1}^{(b)}:=\sum_{j=1, j\neq b}^{N}$. The following local laws for the matrix $T(z,z'):=(T_{ab}(z,z'))_{ab}$ are proved in Section~\ref{sec:two_point_function}.
	\begin{lemma}\label{tracet}
	For all $1 \leq a,b \leq N$ and $z, z' \in D'$ in (\ref{ddd}), we have
		\begin{equation}\label{tracet_locallaw}
		T_{ab}(z,z')=\Big( \frac{m^2_{sc}(z) m^2_{sc}(z') S^2}{1-m_{sc}(z) m_{sc}(z') S} \Big)_{ab} +O_{\prec}\Big(\rho_2 \Psi^\frac{3}{2}(z)\Psi(z')+\rho_2 \Psi(z) \Psi^\frac{3}{2}(z') \Big),
		\end{equation}
where $\rho_2 \equiv \rho_2(z,z'):=(1-m_{sc}(z) m_{sc}(z'))^{-1}$. Moreover, we have the following estimate of the trace of $T(z,z')$:
		\begin{equation}\label{trace_of_T}
		\Tr \Ttwo= \Tr \Big( \frac{m^2_{sc}(z) m^2_{sc}(z')S^2}{1-m_{sc}(z) m_{sc}(z') S}\Big) +\mathcal{E}_{T}(z,z'),
		\end{equation}
		where the error $\mathcal{E}_{T}(z,z')$ is analytic in $z,z' \in \C \setminus \R$ and for all $z,z' \in D'$, it satisfies
		\begin{equation}\label{a41_error}
		|\mathcal{E}_{T}(z,z')| \prec N \Psi^{\frac{3}{2}}(z)\Psi(z')+N \Psi(z)\Psi^{\frac{3}{2}}(z') +N \Theta^2(z)+ N \Theta(z) \Theta(z').
		\end{equation}
		The above results also hold true when $z$ and $z'$ are in different half planes, that is, they also hold true for all $z, \overline{z'} \in D'$.
	\end{lemma}
\begin{remark}
We remark that the above local law is not optimal. If we further expand (\ref{J_2}) in the proof below, the error in (\ref{tracet_locallaw}) can be improved to $O_{\prec}(\rho_2\Psi^2(z)\Psi(z')+\rho_2\Psi(z) \Psi^2(z'))$. But Lemma \ref{tracet} is sufficient to establish the CLTs for the linear statistics, so we do not aim at the optimal local law in the present paper.
\end{remark}
	
Notice that $\Ttwo$ is analytic in $z,z' \in \C \setminus \R$. Taking the partial derivative of $\Tr \Ttwo$ in (\ref{trace_of_T}) and applying the Cauchy integral formula, we have
\begin{align}\label{temp}
\frac{\partial}{\partial z'} \Tr \Ttwo=&\Tr \Big(  \frac{m_{sc}(z) m'_{sc}(z') S}{(1-m_{sc}(z) m_{sc}(z') S)^2}\Big)-m_{sc}(z) m'_{sc}(z')\Tr S\\
&+O_{\prec} \Big( \frac{N \Psi^{\frac{3}{2}}(z)\Psi(z')+N \Psi(z)\Psi^{\frac{3}{2}}(z') +N \Theta^2(z)+ N \Theta(z) \Theta(z')}{|\Im z'|} \Big)\nonumber.
\end{align}

Plugging (\ref{trace_of_T}) (for $z=z'$) and (\ref{temp}) into (\ref{I_1}), we hence obtain that
	\begin{align}\label{I_11}
I_1=& -2 m_{sc}(z) \E [\ea (1-\E)\Tr G]+\frac{ \lambda \Tr S }{2 \pi } \E [\ea] \int_{\Gamma_{2}} \tf(z')  m_{sc}(z) m'_{sc}(z') \dd z'\nonumber\\
	&\quad-\frac{ \lambda}{\pi}  \E [ \ea]  \int_{\Gamma_{2}} \tf(z') \Tr \Big(  \frac{m_{sc}(z) m'_{sc}(z') S}{(1-m_{sc}(z) m_{sc}(z') S)^2}\Big) \dd z'+\mathcal{E}_1(z),
	\end{align}
	where the error $\mathcal{E}_1(z)$ is collected from (\ref{I_1}), (\ref{a41_error}) (for $z=z'$) and (\ref{temp}). Since $g$ is compactly supported, we have $\kappa_0 \leq \kappa \leq C(\kappa_0+\eta_0)$ with $\kappa$ and $\kappa_0$ given in (\ref{kappa}) and (\ref{kappa_0}). Using (\ref{fn}), (\ref{22}), (\ref{rho}) and (\ref{control}), the error $\mathcal{E}_1(z)$ satisfies
	\begin{align}\label{ez}
	|\mathcal{E}_1(z)| \prec & (1+|\lambda|)N^{2 \tau} \bigg( \frac{1}{\sqrt{N \eta_0}}+ \Psi (z)+ N \rho \Psi^3(z)+N \Theta^2(z)+\Theta(z) \eta_0^{-1}+N \Psi^{5/2}(z)\nonumber\\
	&\quad+N \Psi^{\frac{3}{2}}(z) \Big( \frac{(\kappa_0+\eta_0)^{\frac{1}{4}}}{\sqrt{N \eta_0}}+\frac{1}{N \eta_0}\Big)+N \Psi(z) \Big( \frac{(\kappa_0+\eta_0)^{\frac{3}{8}}}{(N \eta_0)^{\frac{3}{4}}}+\frac{1}{(N\eta_0)^{\frac{3}{2}}}\bigg).
	\end{align}

By direct computations and the isotropic local law (\ref{isotropicgene}), one shows that the second term $I_2$ corresponding to third cumulants is negligible (c.f. Section 5.2 \cite{Li+Schnelli+Xu}), 
 \begin{align}\label{I_2}
  I_2  =O_{\prec}\Big( \frac{(1+|\lambda|^2) N^{2\tau} \Psi(z)}{\sqrt{ \eta_0}} \Big)+O_{\prec} \big( \sqrt{N}\Psi^2(z)\big).
  \end{align}
   It is also straightforward to check that (c.f. Section 5.3 and Lemma 5.1 \cite{Li+Schnelli+Xu})
	\begin{align}\label{d2}
         I_3 =&-\frac{k_4  \lambda}{2 \pi} \E \Big[  \ea \int_{\Gamma_{2}}  \tf(z') \pzp  (m^2_{sc}(z) m^2_{sc}(z')) \dd z' \Big] +O_{\prec} \Big( \frac{(1+|\lambda|^3)N^{2\tau}}{ \sqrt{N \eta_0}} \Big),
	\end{align}
	where $k_4$ is the summation of all the fourth cumulants, i.e.,
		\begin{equation}\label{K_4}
		k_4:= \sum_{i,j =1}^N c_{ij}^{(4)}(H_{ij})= \E [H_{ij}^4] -3(\E [H_{ij}^2])^2.
		\end{equation}

	Plugging (\ref{I_11}), (\ref{I_2}) and (\ref{d2}) into (\ref{sum}) and rearranging, we obtain that
	\begin{align}\label{newsum}
	(z+2 m_{sc}(z)) \E [\ea (1-\E)\Tr G] = -\frac{ \lambda}{2 \pi}  \E [ \ea] \int_{\Gamma_{2}} \tf(z') \wt{K}(z,z')\dd z' +\wt{\mathcal{E}} (z),
	\end{align}
	where the kernel $ \wt K(z,z')$ is given by
	\begin{align*}
	\wt{K}(z,z')=2 \Tr \Big(  \frac{m_{sc}(z) m'_{sc}(z') S}{(1-m_{sc}(z) m_{sc}(z') S)^2}\Big)- m_{sc}(z) m'_{sc}(z')\Tr S+2 k_4 m^2_{sc}(z) m_{sc}(z') m'_{sc}(z'),
	\end{align*}
	and $\wt{\mathcal{E}}(z)$ is the error collected from (\ref{ez})-(\ref{d2}).
	Dividing both sides of (\ref{newsum}) by $z+2 m_{sc}(z)=-\frac{m_{sc}(z)}{m'_{sc}(z)} \sim \sqrt{\kappa+\eta}$ from (\ref{msc}) and (\ref{4}), and plugging it into (\ref{phi2}), we hence obtain that the characteristic function satisfies
$$\phi'(\lambda)=-\lambda \E [\ea] V(f)+ O_{\prec} \Big( \frac{(1+|\lambda|^4)N^{4 \tau}}{(N \eta_0 \sqrt{\kappa_0+\eta_0})^{\frac{1}{4}} }\Big)+O_{\prec}\big( |\lambda| N^{ -\tau} \big),$$
where $V(f)$ is given in (\ref{vf}) and we use (\ref{fn}), (\ref{22}) and (\ref{control}) to estimate the error. Note that $|e(\lambda)-\ea| \prec |\lambda| N^{-\tau}$. If $V(f) \prec O(1)$, then we replace $\ea$ by $e(\lambda)$ at the cost of $O_{\prec}(|\lambda|N^{-\tau})$. This completes the proof of Proposition \ref{prop}.
\end{proof}	

We end up this section with the estimate of $\E[\Tr G(z)]-N m_{sc}(z)$ in a similar way and obtain the bias formula for the general linear statistics.

\begin{proof}[Proof of Equation (\ref{bf})] 
 Using the definition of resolvent and the cumulant expansion (\ref{ehh}) on $z \E [\Tr G(z)-N m_{sc}(z)]$, in combination with the local laws in Theorem \ref{locallawgene} and Lemma \ref{tracet}, we have the analogue of (\ref{newsum}), i.e.,
\begin{align}
(z+2m_{sc}(z)) \E (\Tr G(z)-N m_{sc}(z))=&-\Tr \Big( \frac{m^4_{sc}(z)S^2}{1-m^2_{sc}(z)S} \Big)-k_4 m^4_{sc}(z)\\
&\quad+O_{\prec} \big(N\rho \Psi^3\big)+O_{\prec}(N \Psi^{5/2})+O_{\prec}(N\Theta^2)\nonumber,
\end{align}
where $\rho$ is given by (\ref{rho}), and $k_4$ is defined in (\ref{K_4}). Dividing both sides by $z+2m_{sc}(z)$ and transforming $\E [\Tr G(z)]-N m_{sc}(z)$ to the bias of the linear statistics via the Heffler-Str\"osjand formula (\ref{fw2}) and Stokes' formula, we obtain (\ref{bf}) and conclude the last statement of Theorem~\ref{thm:weak_convergence}. 
\end{proof}

\begin{remark}
The results in Section \ref{sec:proof_of_proposition} extend directly from real symmetric $(\beta=1)$ to complex Hermitian $(\beta=2)$ matrices. The only difference is to apply the complex analogue of cumulant expansion formula and $\frac{\partial G_{ij}}{\partial H_{ab}}=-G_{ia}G_{bj}$ instead of (\ref{dH}). Similar arguments can be found in Appendix A \cite{Li+Schnelli+Xu}, and we omit them here.
\end{remark}

\section{Proof of Lemma \ref{tracet}}\label{sec:two_point_function}
In this section, we prove a local law for the two point function $T(z,z')$ defined in (\ref{t_function}). For notational simplicity, we write $T \equiv T(z,z')$, $m_1:=m_{sc}(z)$, $m_2:=m_{sc}(z')$ and define the control parameter $\Xi_2:=\Psi^{\frac{3}{2}}(z) \Psi(z')+\Psi(z) \Psi^{\frac{3}{2}}(z')$. We aim to prove that
	\begin{align}\label{P_matrix}
	P_{ab}:=-\frac{1}{m_1} T_{ab}+ m_2 (S T)_{ab} +m_1m^2_2 (S^2)_{ab}=O_{\prec}( \Xi_2).
	\end{align}
        Due to (\ref{msc}) and the relation $zG=HG-I$, we write
	\begin{align}\label{P_matrix_new}
	P_{ab}=&\sum_{j=1}^{(b)} s_{aj} (HG)_{jb}(z) G_{jb}(z')+m_1 T_{ab}+m_2 (ST)_{ab}+m_1m^2_2 (S^2)_{ab}.
	\end{align}
Set $M_{p,q}:=(P_{ab})^{p} (P^*_{ab})^{q}$ for any $p,q \in \N$ for short. For any $d \in \N$, applying the cumulant expansion (\ref{ehh}), we have
	\begin{align}\label{sum_of_J}
	\E | P_{ab} |^{2d}=&\E \Big[  \Big(  \sum_{j=1}^{(b)} \sum_{k=1}^N s_{aj} s_{jk} \frac{\partial G_{kb}(z) G_{jb}(z')}{\partial H_{jk}}  \Big)  M_{d-1,d}  \Big]\nonumber\\
	&+\E \Big[  \Big(  \sum_{j=1}^{(b)} \sum_{k=1}^N s_{aj} s_{jk} G_{kb}(z) G_{jb}(z')  \Big) (d-1) \frac{\partial P_{ab}}{\partial H_{jk}} M_{d-2,d}  \Big]\\
	&+\E \Big[  \Big(\sum_{j=1}^{(b)} \sum_{k=1}^N s_{aj} s_{jk} G_{kb}(z) G_{jb}(z')  \Big) d \frac{\partial P^*_{ab}}{\partial H_{jk}}M_{d-1,d-1} \Big]+R_2\nonumber\\
	&+\E \Big[  \big( m_1 T_{ab}+m_2 (ST)_{ab}+m_1m^2_2 (S^2)_{ab} \big)   M_{d-1,d} \Big]:=J_1+J_2+J_3+R_2+J_4,\nonumber
	\end{align}
        where $R_2$ is the error of cumulant expansion, see (\ref{error_cumulant}) for $l=1$. We first show that $R_{2}$ is negligible. 
	We write $\ga:=G(z)$, $\gb:=G(z')$, $\Psi_1:=\Psi(z)$ and $\Psi_2:=\Psi(z')$ for short. Using identity (\ref{dH}) and the local law (\ref{G}), for general $\alpha \in \N$, we have 
	\begin{align}\label{dgg}
	\Big| \frac{\partial^\alpha \ga_{kb}\gb_{jb} }{\partial H_{jk}^{\alpha}} \Big| \prec \Psi_1 \Psi_2+\delta_{jb} +\delta_{kb}.
	\end{align}
	We obtain from (\ref{dH}) and (\ref{P_matrix}) that
	\begin{align}\label{dP}
	\frac{\partial P_{ab}}{\partial H_{jk}}=&-\frac{1}{m_1} \frac{\partial T_{ab}}{\partial H_{jk}} +m_2 \sum_{i=1}^N s_{ai} \frac{\partial T_{ib}}{\partial H_{jk}}=-\sum_{l=1}^{(b)} \Big( \frac{1}{m_1}  s_{al}-m_2 (S^2)_{al} \Big)\times\nonumber\\
	&\Big(\ga_{lj} \ga_{kb} \gb_{lb}+\ga_{lk} \ga_{jb} \gb_{lb} +\ga_{lb} \gb_{lj} \gb_{kb} +\ga_{lb} \gb_{lk} \gb_{jb}\Big).
	\end{align}
In general, for any $\alpha \in \N$, the local law (\ref{G}) implies that, 
\begin{align}\label{Perror}
\Big|\frac{\partial^\alpha P_{ab}}{\partial H^\alpha_{jk}}\Big| \prec \Psi_1^2\Psi_2+\Psi_1 \Psi_2^2+\delta_{jb} \Psi_1\Psi_2+\delta_{kb} \Psi_1\Psi_2.
\end{align}
Similarly, the estimate (\ref{Perror}) still holds true for $\frac{\partial^{\alpha} P^*_{ab}}{\partial H_{jk}^\alpha}$ for general $\alpha \in \N$. Using the moment condition (\ref{moment_condition}), (\ref{dgg}), (\ref{Perror}) and (\ref{psibound}), we hence obtain that
	\begin{align}\label{J_4}
	R_2=&\E [O_{\prec}(\Xi_1) M_{d-1,d}] +\E [O_{\prec}(\Xi^2_1 ) M_{d-2,d}] +\E [O_{\prec}(\Xi_1^2) M_{d-1,d-1}]\nonumber\\
	&+\E [O_{\prec}(\Xi^3_1) M_{d-3,d}] +\E [O_{\prec}(\Xi^3_1 ) M_{d-2,d-1}] +\E [O_{\prec}(\Xi_1^3) M_{d-1,d-2}],
	\end{align}
	where we define a new control parameter $\Xi_1:=\Psi^2_1\Psi_2+\Psi_1 \Psi_2^2 \ll \Xi_2$.
	
Next, we look at the first term $J_1$. Using (\ref{dH}) and the local law (\ref{G}), we write
	\begin{align}\label{J_1}
	J_1=&-\E \Big[    \sum_{j=1}^{(b)} \sum_{k=1}^N s_{aj} s_{jk} \Big( m_1 \ga_{jb} \gb_{jb} +m_2 \ga_{kb} \gb_{kb} \Big)  M_{d-1,d}\Big] +\E [O_{\prec}(\Xi_1) M_{d-1,d}]\nonumber\\
	=&-\E \Big[  \big( m_1T_{ab}+m_2 \sum_{j=1}^N s_{aj} T_{jb} +m_1 m^2_2 (S)_{ab}^2\big)   M_{d-1,d} \Big]+\E [O_{\prec}(\Xi_1) M_{d-1,d}].
	\end{align}
	Note that the leading term of $J_1$ will cancel $J_4$.  As for the second term $J_2$, using (\ref{dP}) and simple power counting by the local law (\ref{G}), we have
	\begin{align}\label{J_2}
	J_2=(d-1)\E \Big[  \Big(  \sum_{j=1}^{(b)} \sum_{k=1}^N s_{aj} s_{jk} G_{kb}(z) G_{jb}(z')  \Big) \frac{\partial P_{ab}}{\partial H_{jk}} M_{d-2,d}  \Big]=\E [O_\prec(\Xi_2^2)M_{d-2,d}].
	\end{align}
	We treat $J_3$ similarly and get $J_3= \E [O_\prec(\Xi_2^2)M_{d-1,d-1}]$. Therefore, we obtain that
	\begin{align}\label{moment_2d}
	\E | P_{ab} |^{2d}=&\E [O_{\prec}(\Xi_1) M(d-1,d)] +\E [O_{\prec}(\Xi_2^2) M_{d-1,d-1}] +\E [O_{\prec}(\Xi_2^2) M_{d-2,d}]\nonumber\\
	&+\E [O_{\prec}(\Xi_1^3) M_{d-3,d} ]+\E [O_{\prec}(\Xi^3_1 ) M_{d-2,d-1}] +\E [O_{\prec}(\Xi_1^3) M_{d-1,d-2}].
	\end{align}
	Applying the Young's inequality to the RHS of (\ref{moment_2d}) and using $\Xi_1 \ll \Xi_2$, we get $\E|P_{ab}|^{2d} \prec \Xi_2^{2d}$ for any $d \in \N$ and thus $|P_{ab}| \prec \Xi_2$. Using $|m_{sc}(z)| \sim 1$, the matrix $(T)_{ab}$ defined in (\ref{P_matrix}) hence satisfies	
	\begin{equation}\label{T}
	\Big(1-m_1 m_2 S\Big)T=m^2_1m^2_{2} S^2  +\mathcal{R}(z,z'),
	\end{equation}
	where the error matrix $\mathcal{R} \equiv \mathcal{R}(z,z')$ has the following estimate:
	\begin{align}\label{rz}
	\|\mathcal{R} (z,z')\|_{\sup}=O_{\prec}(\Psi^\frac{3}{2}(z)\Psi(z'))+O_{\prec}(\Psi(z) \Psi^\frac{3}{2}(z')).
	\end{align}
	Combining with the first estimate from Lemma \ref{Pi}, we hence prove (\ref{tracet_locallaw}). 
	
	Next, we continue to estimate the trace of the two point function $T(z,z')$. Recall the projection matrix $\Pi=\mathbf{e}\mathbf{e}^*$, where $\mathbf{e}=N^{-\frac{1}{2}} (1,1, \cdots ,1)^*$. Note that $\Pi S=S \Pi=\Pi.$ Multiplying both sides of (\ref{T}) by $(1-\Pi)(1-m_1 m_2 S)^{-1}$, we have
	$$ (1-\Pi) T=m^2_1m^2_2 \frac{S^2-\Pi}{1-m_1m_2 S} +\frac{1-\Pi}{1-m_1m_2 S} \mathcal{R}.$$
	Using the second estimates in Lemma \ref{Pi} and (\ref{rz}), we obtain that
	\begin{equation}\label{tracesum}
	\Tr T=\Tr (\Pi T)+ \Tr \Big( \frac{m_1^2m_2^2(S^2-\Pi)}{1-m_1m_2 S}\Big)+O_{\prec} \big( N\Psi^{\frac{3}{2}}(z)\Psi(z')\big)+O_{\prec} \big(N \Psi(z)\Psi^{\frac{3}{2}}(z') \big).
	\end{equation}
         For the first term on the RHS of (\ref{tracesum}), we write it as
	\begin{align}\label{tempp}
	\Tr (\Pi T) =\frac{1}{N} \sum_{b=1}^N \sum_{j}^{(b)} G_{jb}(z) G_{jb}(z')= \frac{1}{N}\Tr( G(z) G(z') )- \frac{1}{N}\sum_{b=1}^N G_{bb}(z) G_{bb}(z').
	\end{align}
To estimate (\ref{tempp}), we separate our argument into two cases:
\begin{enumerate}
\item For $z \neq z'$, using the resolvent identity	
\begin{equation}\label{resolvent}
	G(z)G(z')=\frac{1}{z-z'} (G(z)-G(z')),
	\end{equation}
	and the local law (\ref{G}), we have
	$$\Tr (\Pi T) =\frac{m_N(z)-m_{N}(z')}{z-z'}-m_{sc}(z)m_{sc}(z')+O_{\prec}(\Theta(z))+O_{\prec}(\Theta(z')).$$
If $z,z'$ are in different half planes, then $|z-z'| \geq |\Im z |$ and thus from (\ref{G}),
$$\frac{m_N(z)-m_{N}(z')}{z-z'}  =\frac{m_{sc}(z)-m_{sc}(z')}{z-z'}  + O_{\prec}\Big( \frac{ \Theta(z)}{|{\Im}z|}\Big)+O_{\prec}\Big( \frac{ \Theta(z')}{|{\Im}z|}\Big).$$ 
If $z$ and $z'$ are in the same half-plane, without loss of generality, we assume  $z,z' \in \C^+$. If $|{\Im}z-{\Im}z'| \geq \frac{1}{2} {\Im}z$, the previous argument still applies. Otherwise, we have $ \frac{1}{2} {\Im}z \leq {\Im}z' \leq \frac{3}{2} {\Im}z$. Since $d(z):=m_N(z) -m_{sc}(z)$ is analytic in $ z \in \C^+$, applying the Cauchy integral formula, we obtain that
	$$\Big| \frac{d(z)- d(z')}{z-z'} \Big| \leq \sup_{\omega \in L(z,z')} \Big|d'(\omega)\Big| \prec \frac{\Theta(\omega)}{ |{\Im} \omega|}=O_{\prec}\Big( \frac{\Theta(z)}{|{\Im}z|} \Big),$$
	where $L(z,z')$ denotes the segment connecting $z$ and $z'$. Therefore, we have
	\begin{align}\label{T2}
	\Tr (\Pi T) =\frac{m_{sc}(z)-m_{sc}(z')}{z-z'}-m_{sc}(z)m_{sc}(z')+ O_{\prec}\Big( \frac{ \Theta(z)+\Theta(z')}{|{\Im}z|} \Big).
	\end{align}	
	
	\item For $z =z'$, using the identity $G^2(z)=\frac{\dd}{\dd z} G(z)$, the local law (\ref{G}) and the Cauchy integral formula, we have
	\begin{align}\label{T_22}
	\Tr (\Pi T) = \frac{\dd }{\dd z} m_N(z) - \frac{1}{N}\sum_{b=1}^N (G_{bb}(z))^2=m'_{sc}(z)-m^2_{sc}(z)+O_{\prec}\Big( \frac{ \Theta(z)}{|\Im z|}\Big).
	\end{align}
\end{enumerate}

	In addition, we use the Taylor expansion on $(1-m_1m_2S)^{-1}$ and the relation $\Pi S=S\Pi=\Pi$ to get 
	\begin{equation}\label{T1}
	 \Tr \Big( \frac{m_1^2m_2^2\Pi}{1-m_1m_2 S}\Big)=\frac{m^2_1m^2_2}{1-m_1m_2}.
	\end{equation}
	Plugging (\ref{T2}) (or (\ref{T_22}) for $z=z'$)  and (\ref{T1}) into (\ref{tracesum}), we conlucde from (\ref{msc}) that (\ref{trace_of_T})  and (\ref{a41_error}) hold. This complete the proof of Lemma \ref{tracet}.

\section*{Acknowledgements}
Y. Li is supported by the European Research Council (ERC) under the European Union Horizon 2020 research and innovation program, grant 647133 (ICHAOS). Y. Xu is supported by G\"oran Gustafsson Foundation and the Swedish Research Council Grant VR-2017-05195.

\end{document}